\documentclass[11pt,a4paper]{amsart}
\usepackage{hyperref,amsthm,amsfonts,amsmath, amssymb, graphicx,float,graphicx, enumerate, algorithmic, algorithm, color}
\usepackage[numbers,sort&compress]{natbib}
\usepackage[lmargin=34mm,rmargin=34mm,tmargin=34mm,bmargin=34mm]{geometry}

\newcommand{\spacing}[1]{\renewcommand{\baselinestretch}{#1}\setlength{\footnotesep}{\baselinestretch\footnotesep}}
\spacing{1.2}  

\theoremstyle{plain}
\newtheorem{theorem}{Theorem}[section]
\newtheorem{lemma}[theorem]{Lemma}

\theoremstyle{definition}

\newcommand{\N}{\mathbb{N}}
\newcommand{\DEF}{\sl}       

\title{A tighter Erd\H{o}s-P\'osa function for long cycles}
\author{Samuel Fiorini and Audrey Herinckx}
\address{Universit\'e libre de Bruxelles,
Department of Mathematics CP 216,
Boulevard du Triomphe,
B-1050 Brussels,
Belgium.
{\tt \{sfiorini,auherinc\}@ulb.ac.be}}

\begin{document}
\maketitle

\begin{abstract}
We prove that there exists a bivariate function $f$ with $f(k,\ell) = O(\ell \cdot k \log k)$ such that for every naturals $k$ and $\ell$, every graph $G$ has at least $k$ vertex-disjoint cycles of length at least $\ell$ or a set of at most $f(k,\ell)$ vertices that meets all cycles of length at least $\ell$. This improves a result by Birmel\'e, Bondy and Reed (Combinatorica, 2007), who proved the same result with $f(k,\ell) = \Theta(\ell \cdot k^2)$.
\end{abstract}

\section{Introduction}

A collection of graphs $\mathcal{H}$ is said to have the {\DEF Erd\H{o}s-P\'osa} property if there exists a function $f : \mathbb{N} \to \mathbb{R}_+$ such that for every natural $k$ and every graph $G$ at least one of the following two assertions holds:

\begin{itemize}
\item $G$ contains a collection of $k$ vertex-disjoint subgraphs $G_1$, \ldots, $G_k$, each isomorphic to a graph in $\mathcal{H}$;
\item $G$ contains a set $X$ of $f(k)$ vertices such that no subgraph of $G-X$ is isomorphic to a graph in $\mathcal{H}$.
\end{itemize}

A collection $G_1$, \ldots, $G_k$ as above is called a {\DEF packing} and a set $X$ as above is called a {\DEF transversal}. These definitions are motivated by a celebrated result of Erd\H{o}s and P\'osa \cite{EP}. Denoting by $C_t$ the cycle of length $t$, they proved that $\mathcal{H} = \{C_t \mid t \geqslant 3\}$ has the Erd\H{o}s-P\'osa property. They obtain a function $f$ in $\Theta(k \log k)$ and prove that this function $f$ is best possible, up to a constant. 

Our main result is as follows.

\begin{theorem}
\label{thm:main}
There exists a function $f : \mathbb{N}^2 \to \mathbb{R}_+$ with $f(k,\ell) = O(\ell \cdot k \log k)$ such that for every $k, \ell \in \N$ and every graph $G$, at least one of the two following assertions holds:

\begin{itemize}
\item $G$ contains $k$ vertex-disjoint cycles of length at least $\ell$;
\item $G$ contains a set of $f(k,\ell)$ vertices meeting all the cycles of length at least $\ell$.
\end{itemize}
\end{theorem}

This result implies in particular that the collection $\mathcal{H} := \{C_t \mid t \geqslant \ell\}$ has the Erd\H{o}s-P\'osa property for each fixed natural $\ell$. Birmel\'e, Bondy and Reed \cite{BBR07} proved Theorem \ref{thm:main} with $f(k,\ell) = \Theta(\ell \cdot k^2)$ and left as an open problem to find the correct order of magnitude of $f$. Theorem \ref{thm:main} essentially settles this problem. Our function $f$ is best possible up to a constant for each fixed $\ell$. Moreover, it is also best possible up to a constant for each fixed $k$. However, we do not known whether it is best possible up to a constant when both $k$ and $\ell$ vary.

Before giving the outline of this paper, we mention a few relevant references concerning the case where $\mathcal{H}$ consists of all the graphs containing a fixed graph $H$ as a minor. Robertson and Seymour have shown that $\mathcal{H}$ has the Erd\H{o}s-P\'osa property if and only if $H$ is planar~\cite{RS_five}. They left wide open the problem of determining the order of magnitude of the best possible function $f$, for each fixed planar graph $H$. Our main result answers this problem when $H$ is a cycle. A recent paper of Fiorini, Joret and Wood~\cite{FJW12} answers the problem when $H$ is a forest. In this case, it turns out that $f$ can be taken to be linear in $k$.

The outline of the rest of the paper is as follows. We begin with some preliminaries in Section \ref{sec:preliminaries}. The proof of Theorem
\ref{thm:main} is given in Section \ref{sec:proof}.

\section{Preliminaries} \label{sec:preliminaries}

Before proving our main result, we state a few lemmas that are used in the proof. For $k \in \N$, let 
$$
s_k := 
\begin{cases}
4k \log k + 4k \log \log k + 16k &\text{if } k \geqslant 2\\
1 &\text{if } k \leqslant 1\ .
\end{cases}
$$
Notice that $s_k = \Theta(k \log k)$.

\begin{lemma}[Erd\H{o}s and P\'osa \cite{EP}, Diestel \cite{DiestelBook}]
\label{lem:cubic}
For every $k \in \N$, every cubic multigraph $H$ with at least $s_k$ vertices contains $k$ vertex-disjoint cycles. 
\end{lemma}

Let $\ell \in \mathbb{N}$ be fixed. Below, we call a cycle {\DEF long} if its length is at least $\ell$, and {\DEF short} otherwise. Our proof relies on the following lemma (see below). Birmel\'e, Bondy and Reed conjecture that the lemma still holds when $2\ell+3$ is replaced by $\ell$, which would be tight.

\begin{lemma}[Birmel\'e, Bondy and Reed \cite{BBR07}]
\label{lem:nu_eq_1}
If a graph $G$ does not contain two vertex-disjoint long cycles, then it contains a set of at most $2\ell + 3$ vertices that meets all the long cycles.
\end{lemma}

Compared to the two previous lemmas, our next lemma is rather obvious. We nevertheless include a proof for completeness.

\begin{lemma}
\label{lem:disjoint_witnesses}
Let $z$ and $z'$ be two distinct vertices of $G$. Let $C_z$ and $C_{z'}$ denote two cycles of $G$ of length at least $2\ell$ containing $z$ and $z'$, respectively. If $C_z$ and $C_{z'}$ are not disjoint, then $C_z \cup C_{z'}$ contains a $z$--$z'$ path of length at least $\ell$.
\end{lemma}

\begin{proof}
Follow $C_z$ in any direction from $z$ until the first vertex of $C_{z'}$ is reached, say $t$. One of the two $t$--$z'$ paths in $C_{z'}$ has length at least $\ell$. Thus the desired $z$--$z'$ path exists.
\end{proof}

\section{The proof} \label{sec:proof}

\begin{proof}[Proof of Theorem \ref{thm:main}]
We prove the theorem with $f(k) := (2 \ell + 4)(k-1) + (3\ell/2 + 1)s_k$, by induction on $k$. For $k \leqslant 1$, the theorem obviously holds. From now on, we assume $k \geqslant 2$. Because $f(k-1) + 2\ell \leqslant f(k)$, we may assume without loss of generality that $G$ does not contain any cycle of length between $\ell$ and $2 \ell$. 

Let $H$ denote a subgraph of $G$ with the following properties:
\begin{enumerate}[(i)]
\item all vertices $v$ of $H$ have degree $2$ or $3$ in $H$;
\item $H$ contains no short cycle;
\item the size of $U := \{v \in V(H) \mid \deg_H(v) = 3\}$ is maximum. \end{enumerate}
Notice that all components of $H-U$ are either cycles or paths. Among all subgraphs $H$ satisfying (i)--(iii), we choose one such that the number of cycles in $H-U$ is maximum. 

Let $U'$ denote the set of vertices of $H$ whose distance in $H$ to $U$ is at most $\ell/2$. In a formula,
$$
U' := \{v \in V(H) \mid d_H(U,v) \leqslant \ell/2\}\ .
$$

Consider a $H$-path $P$ that avoids $U'$. We claim that the endpoints of $P$ are in the same component of $H-U$. Suppose that $P$ has its endpoints in two different components of $H-U$. If one of these components is a cycle, then $H + P$ satisfies (i)--(ii) but has two more degree-$3$ vertices, a contradiction. If the two components are paths, then $H+P$ always satisfies (i), and also satisfies (ii) unless a short cycle appears when $P$ is added to $H$. This short cycle intersects $U$. In particular, one of the two endpoints of $P$ is at distance at most $\ell/2$ from $U$. Hence $P$ is not disjoint from $U'$, in contradiction with our hypotheses. Therefore, our claim holds.

Now consider a long cycle $C$ that avoids $U'$.

By choice of $H$, this cycle $C$ has some vertex in $H$, because otherwise we could add $C$ to $H$ and increase the number of cycles in $H-U$ without changing the size of $U$, contradicting our choice of $H$. 

It could well be that $C$ meets $H$ in exactly one vertex. For now, we assume that $C$ contains at least two vertices of $H$. By the above claim, $C$ meets exactly one component of $H-U$, say $K$. Let $K'$ denote the subgraph of $G$ obtained by adding to $K$ all the $H$-paths $P$ with both endpoints in $K$. Then $C$ is contained in $K'$.\medskip

\noindent \emph{Case 1.} $K$ is a cycle. We claim that $K'$ does not contain two vertex-disjoint long cycles. Otherwise, we could redefine $H$ by replacing $K$ by these two cycles and contradict the choice of $H$. By Lemma \ref{lem:nu_eq_1}, $K'$ contains a set of at most $2 \ell + 3$ vertices that meets all the long cycles in $K'$.\medskip

\noindent \emph{Case 2.} $K$ is a path. We claim that this case cannot occur. Indeed, let $u$ and $v$ denote the two endpoints of $K$. We can redefine $H$ by replacing $K$ by the subgraph of $K'$ obtained from the long cycle $C$ by connecting $u$ and $v$ to this cycle through $K$. Because two of the vertices of $K$ are in $C$, and $C$ avoids $U'$, this operation preserves properties (i) and (ii). But this operation increases the number of vertices in $U$, and thus contradicts our choice of $H$.

\medskip

Now to the conclusion. Let $\mathcal{C}$ denote the set of long cycles that avoid $U$ and meet $H$ in exactly one vertex. Let $Z$ denote the set of vertices of $H$ that are in some cycle of $\mathcal{C}$. For each vertex $z \in Z$, pick a {\DEF witness} cycle $C_z$ in $\mathcal{C}$. Any two distinct witness cycles $C_z$ and $C_{z'}$ are disjoint because otherwise, by Lemma \ref{lem:disjoint_witnesses}, we could find in their union a $H$-path with endpoints $z$ and $z'$ of length at least $\ell$. Adding such an $H$-path to $H$ does not create a short cycle and increases the number of vertices in $U$, a contradiction. Therefore, 
$$
\mathcal{C}' := \{C_z \mid z \in Z\}
$$
is a collection of vertex-disjoint long cycles.

Let $\mathcal{D}$ denote the set of cycles in $H-U$ that are disjoint from $Z$. Thus, $\mathcal{C}' \cup \mathcal{D}$ is also a collection of vertex-disjoint long cycles. 

If $\mathcal{C}' \cup \mathcal{D}$ contains at least $k$ cycles, then the theorem holds. Assume now that $|\mathcal{C}' \cup \mathcal{D}| \leqslant k-1$.

Now that we have failed to produce a large packing of long cycles, we try for a small transversal of the long cycles. The transversal is obtained as follows:

\begin{itemize}
\item in each component $K$ of $H-U$ that is a cycle, pick a set of at most $2 \ell + 3$ vertices that meet all the long cycles in $K'$ (by Lemma \ref{lem:nu_eq_1}, such a set exists);
\item add all the vertices of $Z$;
\item add all the vertices of $U'$.
\end{itemize}

The total number of vertices in the transversal is bounded by
\begin{eqnarray*}
(2 \ell + 3) (|\mathcal{C}' \cup \mathcal{D}|) + |Z| + |U'|
&\leqslant & (2 \ell + 4) (k-1) + |U'|\\
&\leqslant & (2 \ell + 4) (k-1) + (3\ell/2 + 1) |U|\ .
\end{eqnarray*}
If $|U| \leqslant s_k$ then the theorem holds. Otherwise, we have $|U| \geqslant s_k$ and by Lemma \ref{lem:cubic} applied to the multigraph obtained from $H$ by suppressing all degree-$2$ vertices, $H$ contains $k$ vertex-disjoint cycles. By choice of $H$, these cycles are all long. Thus the theorem holds in this case also.
\end{proof}

\section{Acknowledgements}

We thank Gwena\"el Joret for stimulating discussions and feedback.

\bibliographystyle{plain}

\bibliography{long_cycles}

\end{document}